\documentclass[12pt]{amsart}
\usepackage{hyperref}
\usepackage{amsfonts}
\usepackage{bbm}
\usepackage{esint}
\usepackage{mathrsfs}
\usepackage{amssymb,amsmath,amsfonts,amsthm}

\numberwithin{equation}{section}

\DeclareMathOperator{\supp}{supp}

\allowdisplaybreaks

\begin{document}

\baselineskip 16.1pt \hfuzz=6pt

\theoremstyle{plain}
\newtheorem{theorem}{Theorem}[section]
\newtheorem{prop}[theorem]{Proposition}
\newtheorem{lemma}[theorem]{Lemma}
\newtheorem{corollary}[theorem]{Corollary}
\newtheorem{example}[theorem]{Example}
\newtheorem*{thmA}{Theorem A}
\newtheorem*{thmB}{Theorem B}
\newtheorem*{thmC}{Theorem C}
\newtheorem*{defnA}{Definition A}
\newtheorem*{defnB}{Definition B}
\newtheorem*{defnC}{Definition C}
\newtheorem*{defnD}{Definition D}
\newtheorem*{defnE}{Definition E}

\theoremstyle{definition}
\newtheorem{definition}[theorem]{Definition}
\newtheorem*{remark}{Remark}

\renewcommand{\theequation}
{\thesection.\arabic{equation}}

\allowdisplaybreaks
\newcommand{\lc}{\rule{0pt}{8pt}^c}
\newcommand{\XX}{X}
\newcommand{\GG}{\mathop G \limits^{    \circ}}
\newcommand{\GGs}{{\mathop G\limits^{\circ}}}
\newcommand{\GGtheta}{{\mathop G\limits^{\circ}}_{\theta}}
\newcommand{\xoneandxtwo}{X_1\times\mathcal X_2}

\newcommand{\GGp}{{\mathop G\limits^{\circ}}}

\newcommand{\GGpp}{{\mathop G\limits^{\circ}}_{\theta_1,\theta_2}}

\newcommand{\e}{\varepsilon}
\newcommand{\bmo}{{\rm BMO}}
\newcommand{\vmo}{{\rm VMO}}
\newcommand{\cmo}{{\rm CMO}}
\newcommand{\Z}{\mathbb{Z}}
\newcommand{\N}{\mathbb{N}}
\newcommand{\R}{\mathbb{R}}
\newcommand{\C}{\mathbb{C}}
\newcommand{\Rn}{\mathbb{R}^n}

\def\S{{\bf S}}
\def\xia {{\xi^{(1)}}}
\def\xib {{\xi^{(2)}}}
\def\xin {{\xi^{(n)}}}
\def\Phia {\Phi^{(1)}}
\def\Phib {\Phi^{(2)}}
\def\Phin {\Phi^{(n)}}
\def\muv {{\mu,v}}
\def\a {{^{(1)}}}
\def\b {{^{(2)}}}
\def\na {{n_1}}
\def\nb {{n_2}}
\def\Ia {{I_1}}
\def\Ib {{I_2}}
\def\Ja {{J_1}}
\def\Jb {{J_2}}
\def\tona {{1,2,\cdots,n_1}}
\def\tonb {{1,2,\cdots,n_2}}
\def\ma {{m_1}}
\def\mb {{m_2}}
\def\H{{\bf H}}
\def\L{{\bf L}}
\def\mt{\longrightarrow}
\def\bel{\begin{eqnarray}\label}
\def\eeq{\end{eqnarray}}
\def\F{\mathfrak{F}}
\def\etaa{{\eta^{(1)}}}
\def\etab{{\eta^{(2)}}}
\def\bel{\begin{eqnarray}\label}
\def\eeq{\end{eqnarray}}
\def\p{{\partial}}

\pagestyle{myheadings}\markboth{\rm\small
}{\rm\small Boundedness of singular integrals }

\title[New Type Fourier Integral Operators]
{Boundedness of New Type Fourier Integral Operators with Product Structure}

\author{{Chaoqiang Tan and Zipeng Wang}}



\begin{abstract}
We investigate a class of Fourier integral operators with weakened symbols, which satisfy a multi-parameter differential inequality in $\R^n$. We establish that these operators retain the classical $L^p$ boundedness and the $H^1$ to $L^1$ boundedness. Notably, the Hardy space considered here is the traditional single-parameter Hardy space rather than a product Hardy space.
\end{abstract}
\maketitle

\section{Introduction}\label{sec:introduction}
\setcounter{equation}{0}
Fourier integral operators are a crucial focus of harmonic analysis and find widespread applications in partial differential equations. They are particularly instrumental in solving initial value problems associated with wave equations, where the solution is often expressed through the corresponding Fourier integral operator derived from the initial data. 

A Fourier integral operator (FIO) is defined by:
\begin{eqnarray}\label{F}
\F f(x)=\int_{\R^n}e^{2\pi i \Phi(x,\xi)}\sigma(x,\xi)\hat{f}(\xi)d\xi,
\end{eqnarray}
where $\hat{f}(\xi)=\int_{\R^n}f(x)e^{-2\pi i x\cdot \xi}d\xi$.
The symbol function $\sigma(x,\xi)\in C^\infty(\R^n\times\R^n)$, and its support $\supp(\sigma)\subseteq E\times \R^n$, where $E$ is a compact set in $\mathbb{R}^n$. The phase $\Phi$ is  real-valued, homogeneous of degree $1$ in $\xi$ and smooth in $E\times (\R^n\setminus 0)$. Moreover, it satisfies   
 the non-degeneracy condition
\begin{eqnarray*}
 \det\left[{\p^2\Phi\over \p x_i\p \xi_j}\right]\left(x,\xi\right)~\neq~0,
 \end{eqnarray*}
on $E\times (\R^n\setminus 0)$.  

The foundational theory of Fourier integral operators (FIOs) on $\R^n$ was initiated and thoroughly developed by H\"ormander \cite{Hormander}. In his seminal work, he introduced the symbol class $S^m$ of order $m$, denoted as:
\bel{class1}
\left|\p_\xi^\alpha\p_{x}^\beta \sigma(x,\xi)\right| \leqslant C_{\alpha,\beta} (1+|\xi|)^{m-|\alpha|}
\eeq
for all multi-indices $\alpha$ and $\beta$.

For symbols $\sigma\in S^0$, H\"ormander demonstrated in \cite{Hormander} that the corresponding FIO $\F$ is bounded on $L^2(\R^n)$. This result can also be found in the work of Eskin \cite{Eskin}.In contrast to the $L^2$ result mentioned above, it is well-established that an FIO of order zero does not remain bounded on $L^p(\R^n)$  spaces for values of $p$ different from 2. The optimal $L^p$ estimate for FIOs  was first explored by Duistermaat and H\"{o}rmander \cite{Duistermaat-Hormander}, followed by investigations by Colin de Verdi\'{e}re and Frisch \cite{Colin-Frisch},  Brenner \cite{Brenner}, Peral \cite{Peral}, Miyachi \cite{Miyachi}, Beals \cite{Beals}, culminating in the iconic results obtained by Seeger, Sogge, and Stein\cite{S.S.S}.

\begin{theorem}[Seeger, Sogge and Stein]
 Let $\F$ be a Fourier integral operator defined in \eqref{F}. Assume that the symbol $\sigma\in S^m$ for some $-\frac{n-1}{2}\leqslant m\leqslant 0.$  We then have the following result:\\
{\rm (i)} If $m=-{n-1\over 2},$ then both $\F$ and its adjoint $\F^*$ map $H^1(\R^n)$ to $L^1(\R^n)$.\\
{\rm (ii)} If $-\frac{n-1}{2}< m\leqslant 0$ and $|{1\over p}-{1\over 2}|\leqslant{-m\over n-1},$ then\ $\F:L^p(\R^n)\to L^p(\R^n).$
\end{theorem}

\begin{remark}
The $L^p$ boundedness of $\F$ above is sharp. For a detailed discussion, please refer to section {\bf 6.13}, chapter IX in Stein's book \cite{Stein}.
\end{remark}

A fundamental premise of our paper is to relax the conditions on the symbol $\sigma$ and examine the boundedness properties of $\F$. To this end, we consider the symbol conditions with a product structure as follows.
\begin{definition}
We denote $\sigma(x,\xi)\in \S^m,$ when
\bel{calss2}
|\partial_{\xi}^\alpha \partial_x^\beta\sigma(x,\xi)|\leqslant C_{\alpha,\beta} (1+|\xi|)^m \prod\limits_{i=1}^n\frac1{(1+|\xi_i|)^{\alpha_i}}.
\eeq
for all multi-indices $\alpha$ and $\beta$.
\end{definition}

\begin{remark}
{\rm (a)} It is evident that  $\sigma\in S^m\Rightarrow \sigma\in \S^m,$ but the converse is not true.
For instance, if we select a specific $\sigma(x,\xi)\in S^m,$ and define $\tilde{\sigma}(x,\xi)=\sigma(x,\xi)\frac{\xi_1^2}{1+\xi_1^2},$ generally $\tilde{\sigma}\not\in S^m$ (for dimensions $n \geqslant 2$),but $\tilde{\sigma}\in \S^m.$

{\rm (b)} The exploration of operators compatible with multi-parameters traces its roots back to the era of Jessen, Marcinkiewicz, and Zygmund. Over the past decades, significant advancements have been made in this field. Notable contributions include those by R. Fefferman \cite{R.Fefferman}, Fefferman and Stein \cite{R-F.S}, Cordoba and Fefferman \cite{Cordoba-Fefferman}, Chang and Fefferman \cite{Chang-Fefferman}, and M\"{u}ller, Ricci, and Stein \cite{M.R.S}.
\end{remark}

With the aforementioned preparations in place, we are now poised to present the main results of this paper.
\begin{theorem}[{\bf Main result}]\label{th1}
 Let $\F$  be a Fourier integral operator defined in \eqref{F}. Assume that the symbol $\sigma\in \S^m$ for some $-\frac{n-1}{2}\leqslant m\leqslant 0.$ We then have the following result:\\
{\rm (i)} If $m=-{n-1\over 2},$ then both $\F$ and its adjoint $\F^*$ map $H^1(\R^n)$ to $L^1(\R^n)$.\\
{\rm (ii)} If $-\frac{n-1}{2}< m\leqslant 0$ and $|{1\over p}-{1\over 2}|\leqslant{-m\over n-1},$ then\ $\F:L^p(\R^n)\to L^p(\R^n).$
\end{theorem}

\begin{remark}
{\rm (a)} For part {\rm (i)} of Theorem \ref{th1}, our proof follows the framework outlined in Stein's book \cite{Stein}. However, due to the weaker conditions for the symbol function $\sigma$, we need to perform a more intricate decomposition of the operator's kernel. The details can be found in the subsequent proof of Lemma \ref{lem-tan}.

{\rm (b)} Surprisingly, when $\sigma\in\S^{-{n-1\over 2}}$, even though the symbol function $\sigma$  has an estimation with a product structure, Theorem \ref{th1} demonstrates that $\F$ is bounded from the single-parameter $H^1(\R^n)$ to $L^1(\R^n)$. This specific single-parameter $H^1(\R^n)$ space was introduced by Fefferman and Stein \cite{FC.S} and characterized by an atomic decomposition by Coifman \cite{Coifman}. This indicates that we have provided an example with a certain product structure that remains bounded from the single-parameter $H^1(\R^n)$ to $L^1(\R^n)$.
\end{remark}

The structure of this paper is as follows: In Section 2, we present the lemmas required for proving Theorem\ref{th1}. In Section 3, we first provide more detailed cone decomposition and notation. Then, assuming Lemma \ref{lem-tan} holds, we present the proof of Theorem \ref{th1}. In Section 4, we prove Lemma \ref{lem-tan}. Because our symbol conditions are weakened, a more detailed decomposition of the kernel function is needed here.

\section{Preliminary Lemma}
Before proceeding to establish Theorem \ref{th1}, we first prove the following lemma.
\begin{lemma} \label{Lemma One}
Assume that $\sigma\in\S^m$ with  $-\frac{n}{2}<m\leqslant0$. We have\\
{\rm (i)} $\F :L^2(\R^n)\to L^2(\R^n).$\\
{\rm (ii)} $\F :L^p(\R^n)\to L^2(\R^n),$ if $\frac1{p}=\frac12-\frac{m}{n}.$\\
{\rm (iii)} $\F :L^2(\R^n)\to L^q(\R^n),$ if $\frac1{q}=\frac12+\frac{m}{n}.$\\
\end{lemma}

\begin{proof}
{\bf step 1: Proof of Lemma\ref{Lemma One}'s part (i)}

Employing the Plancherel theorem streamlines our task, reducing it to considering the operator
\begin{eqnarray*}
T f(x)=\int_{\R^n} e^{2\pi i\Phi(x,\xi)} \sigma(x,\xi) f(\xi) d\xi
\end{eqnarray*}
whose adjoint takes the form
\begin{eqnarray*}
T^* f(\xi)=\int_{\R^n} e^{-2\pi i\Phi(x,\xi)} \bar{\sigma}(x,\xi) f(x) dx.
\end{eqnarray*}

Now we define a notion of an {\it narrow cone} as follows: If vectors $\xi$ and $\eta$ lie within the same narrow cone, with $|\eta|\leqslant|\xi|$, we express $\eta=\rho\xi+\eta^\dagger$, where $0\leqslant\rho\leqslant1$ and $\eta^\dagger$ is orthogonal to $\xi$. A key characteristic of this construction is the requirement that $|\eta^\dagger|\leqslant c\rho|\xi|$, where $c$ is a small positive constant related to the phase function $\Phi$.

Evidently, the frequency domain can be partitioned in such a way that both $T$ and $T^*$ can be represented as finite combinations of partial operators. Each of these partial operators features a symbol whose support resides within one of these narrowly cones. Importantly, this decomposition does not interfere with the differentiation characteristics of $\sigma(x,\xi)$ with respect to $x$.

Recalling the estimate presented in Section {\it 3.1.1}, Chapter IX of Stein's book \cite{Stein}, we are reminded that
\bel{Phi_x est}
\left|\nabla_x \Big(\Phi(x,\xi)-\Phi(x,\eta)\Big)\right|\geqslant C |\xi-\eta|,
\eeq
whenever vectors $\xi$ and $\eta$ are contained within the same narrow cone as previously defined.

\begin{remark}
In addressing the $L^2$-boundedness of the operator $T$ for symbols $\sigma\in\S^m$, our analysis confines the support of $\sigma(x,\xi)$ to lie within one of these narrow cones. Furthermore, it is crucial to note that the symbol $\sigma$ is only required to adhere to the differential inequalities concerning the $x$-variable.
\end{remark}
We now turn our attention to the composition of operators,
\begin{eqnarray*}
T^*T f(\xi)=\int_{\R^n} f(\eta)K^\sharp(\xi,\eta)d\eta,
\end{eqnarray*}
where the associated kernel is defined as
\begin{eqnarray*}
K^\sharp(\xi,\eta)=\int_{\R^n} e^{2\pi i\left(\Phi(x,\eta)-\Phi(x,\xi)\right)} \sigma(x,\eta)\bar{\sigma}(x,\xi) dx.
\end{eqnarray*}

Given that $\sigma(x,\xi)$ enjoys a compact support in the $x$-domain, recalling the estimate (\ref{Phi_x est}), we may employ an $N$-fold integration by parts with respect to $x$ to yield
\begin{eqnarray*}
\begin{array}{lr}
\left|K^\sharp(\xi,\eta)\right|\lesssim 
\left|\xi-\eta\right|^{-N}
\left|\int_{\R^n} e^{2\pi i\left(\Phi(x,\eta)-\Phi(x,\xi)\right)} \nabla_x^N\Big(\sigma(x,\eta)\bar{\sigma}(x,\xi)\Big)dx\right|,
\end{array}
\end{eqnarray*}
whenever  $\xi\neq\eta$. 

Consequently, we arrive at the estimation
\begin{eqnarray*}
\left|K^\sharp(\xi,\eta)\right|\leqslant C \left({1\over 1+|\xi-\eta|}\right)^{N},
\end{eqnarray*}
which holds for any $N \geqslant 1.$
To establish the $L^2$-boundedness of $T$, we proceed by selecting $N > n$ and deduce that

\begin{eqnarray*}
\left\|T^*T f\right\|_{L^2(\R^n)}
&\lesssim& \Big\| \int_{\R^n} \left|f(\cdot-\eta)\right| \left({1\over 1+|\eta|}\right)^{N} d\eta \Big\|_{L^2(\R^n)}
\\
&=& \left\| f\right\|_{L^2(\R^n)}\int_{\R^n}  \left({1\over 1+|\eta|}\right)^{N}d\eta
\\
&\lesssim& \left\| f\right\|_{L^2(\R^n)},
\end{eqnarray*}

With the $L^2$-boundedness of $T$ firmly established, our attention now shifts towards verifying part {\rm (ii)} of Lemma\ref{Lemma One}.

{\bf step 2: Proof of Lemma\ref{Lemma One}'s part (ii)}

For the case $m=0$, step 1 has already provided the proof, so we can assume that $-\frac{n}{2}<m<0.$ Let $\F =\F _0\F_1,$ where $\F_0f(x)=\int_{\R^n}e^{2\pi i \Phi(x,\xi)}\sigma_0(x,\xi)\hat{f}(\xi)d\xi,$ with $\sigma_0(x,\xi)=\sigma(x,\xi)\cdot (1+|\xi|^2)^{-\frac{m}{2}}\in \S^0.$ Meanwhile, $\F_1f(x)=\int_{\R^n}e^{2\pi i \Phi(x,\xi)}(1+|\xi|^2)^{\frac{m}{2}}\hat{f}(\xi)d\xi.$ Firstly, by utilizing $\sigma_0\in \S^0$ and step 1, it is derived that $\F_0:L^2(\R^n)\to L^2(\R^n).$ Therefore, it only remains to prove that $\F_1:L^p(\R^n)\to L^2(\R^n).$

Note that $\F_1f(x)=\int_{\R^n} K(x-y)f(y)dy,$ where the kernel $K$ satisfies $|K(z)|\leqslant A|z|^{-n-m}.$ (see section {\bf 4}, chapter VI in Stein's book \cite{Stein}.)
By applying the Hardy-Littlewood-Sobolev inequality, namely the following inequality (see section {\bf 4.2}, chapter VIII in Stein's book \cite{Stein}).
$$
\|f\ast |y|^{-\gamma}\|_{L^q(\R^n)}\lesssim \|f\|_{L^p(\R^n)},
\ \text{for}\ 0<\gamma<n,1<p<q<\infty \ \text{and}\ \frac1q=\frac1p-\frac{n-\gamma}{n}.
$$
it follows that
$\F_1:L^p(\R^n)\to L^2(\R^n),$ for $\gamma=n+m,q=2$ and $\frac12=\frac1p-\frac{-m}{n}.$

{\bf step 3: Proof of Lemma\ref{Lemma One}'s part {\rm (iii)}}

For the case $m=0$, step 1 has already provided the proof, so we can assume that $-\frac{n}{2}<m<0$. Let $T f(x) = \int_{\R^n} e^{2\pi i\Phi(x,\xi)} \sigma(x,\xi) f(\xi) d\xi,$ then we have $\F=T\mathcal F,$ where $\mathcal F f(\xi)=\hat{f}(\xi)=\int_{\R^n}f(x)e^{-2\pi i x\cdot \xi}d\xi.$ Since $\F^*=\mathcal F^* T^*,$ we only need to prove that $T^*:L^p\to L^2,$ where $\frac1p+\frac1q=1.$

Because 
$$
\|T^*(f)\|_{L^2}^2=\langle T^*f,T^*f\rangle =\langle TT^*f,f\rangle\leqslant \|TT^*f\|_{L^q}\|f\|_{L^p},
$$
we only need to prove that $TT^*:L^p\to L^q.$ Following the approach in step 1, we can presume that the support of $\sigma(x,\xi)$ is confined within a narrow cone.

Since $\triangledown_\xi [\Phi(x,\xi)-\Phi(y,\xi)]=\Phi_{x,\xi}(x,\xi)\cdot (x-y)+O(|x-y|^2),$ we only need the support of $\sigma$ to be sufficiently small, and for any $\xi\ne 0,$ $|\triangledown_\xi [\Phi(x,\xi)-\Phi(y,\xi)]|\geqslant c|x-y|,$ when $x,y\in \supp(a(\cdot,\xi)).$

We can also replace $\sigma(x,\xi)$ with $a_\varepsilon (x,\xi)=\gamma(\varepsilon \xi)\cdot \sigma(x,\xi),$ where $\gamma\in C_0^\infty(\R^n)$ and $\gamma(0)=1.$ This allows us to assume that $\sigma(x,\xi)$ has compact support, noting that the inequality to be estimated below should not depend on $\varepsilon$, i.e., it should not depend on the support of $\sigma$ about $\xi$.

Notice that $TT^*f(x)=\int_{\R^n} K(x,y)f(y)dy,$ where $K(x,y)=\int_{\R^n}e^{2\pi i [\Phi(x,\xi)-\Phi(y,\xi)]}\sigma(x,\xi)\bar{\sigma}(y,\xi)d\xi.$
Using the technique from section {\bf 3.1.4}, chapter IX in Stein's book \cite{Stein}, it can be proven that $|K(x,y)|\lesssim |x-y|^{-n-2m}.$
By applying the Hardy-Littlewood-Sobolev inequality, we have $\|TT^*f\|_{L^q(\R^n)}\lesssim \|f\|_{L^p(\R^n)},$
where $\gamma=n+2m,\frac1{q}=\frac12+\frac{m}{n}$ and $\frac1{p}=\frac12-\frac{m}{n}.$
Lemma\ref{Lemma One} is thus proved.
\end{proof}

\section{Proof of Theorem\ref{th1}}

Now we can proceed to prove Theorem \ref{th1}'s Part {\rm (i)}.
Since a smaller $m$ yields more advantageous estimates for $\sigma$, it suffices to show that when $\sigma\in \S^{-\frac{n-1}{2}}$, both $\F$ and $\F^*$ map $H^1(\R^n)$ to $L^1(\R^n)$.

Let us begin with proving that $\F:H^1(\R^n) \to L^1(\R^n).$ According to the atomic decomposition theorem of Fefferman and Stein \cite{FC.S}, it is only necessary to verify:
$$\int_{\R^n} |\F a(x)|dx\lesssim 1,$$
where $a$ is an atom supported on a ball $B$ of radius $r$ centered at $x_0$, satisfying $\int_{\R^n}a(x)dx=0$ and $\|a\|_\infty \leqslant \frac1{|B|}.$

If $r\geqslant 1,$ leveraging $\sigma\in  \S^{-\frac{n-1}{2}} \Rightarrow {\sigma\in  \S^0},$  and considering the support conditions of $\sigma$ along with the operator $\F$ being $L^2$-bounded (as shown in Lemma \ref{Lemma One} above), we have
\begin{eqnarray*}
\int_{\R^n}|\F (a)|dx&\lesssim& \|\F (a)\|_2 \lesssim \|a\|_2 
\lesssim \frac1{r^{n/2}}\lesssim 1.
\end{eqnarray*}

Thus, we can assume that $r<1$. 
Before proceeding to the next step, we introduce a notation convention: Let $\Omega\subseteq \R^n$ and $J\subset \{1,2,\cdots,n\}.$ Define
$$\Omega^J=\{\xi\in \Omega:\xi_i=0,\forall i\in J \text{ and } \xi_i\ne 0,\forall i\not\in J\},$$
which signifies that for $\xi \in \Omega^J$, the set of coordinates where $\xi$ is zero precisely matches $J$.

In light of the necessity for partitioning directions in our proof, we preliminary fix a collection of sets $\Omega_j=\{\xi_j^v\}_v \subset {\mathbb S}^{n-1}$(the unit sphere), $j \geqslant 0$. These sets are required to fulfill the following conditions:\\
{\rm (i)}$|\xi_j^v-\xi_j^{v'}|\geqslant 2^{-\frac{j}{2}-2},$ if $v\ne v'.$\\
{\rm (ii)} $\bigcup\limits_{i=1}^n \{e_i,-e_i\}\subseteq \Omega_j,$ where $e_i=(0,\cdots,0,1,0,\cdots,0)^T,$ with the $1$ located at the $i$-th component.\\
{\rm (iii)} $\forall 1\leqslant i\leqslant n,$ $(\xi_j^v)_i=0$ or $|(\xi_j^v)_i|\geqslant 2^{-\frac{j}{2}}.$ \\
{\rm (iv)} For every $\xi\in {\mathbb S}^{n-1}$, let $J=\{i:1\leqslant i\leqslant n, |\xi_i|< 2^{-\frac{j}{2}}\},$ then there exists a $\xi_j^v\in \Omega_j^J,$ such that $\sum\limits_{i:i\in \{1,2,\cdots,n\}\setminus J}|(\xi-\xi_j^v)_i|^2<2^{-j-4}.$

\begin{remark}
{\rm (a)} It can be observed that $\Omega_j=\bigcup\limits_{J\subsetneq \{1,2,\cdots,n\}}\Omega_j^J.$

{\rm (b)} If $\xi_j^v\in \Omega_j^J,$ then $|(\xi_j^v)_i| \geqslant  2^{-\frac{j}{2}},$ $\forall i\in \{1,2,\cdots,n\}\setminus J$ and it equals zero for all $i$ in $J$.

{\rm (c)} The cardinality $|\Omega_j^J|\lesssim 2^{\frac{j(n-1-|J|)}{2}}$ and $|\Omega_j|\lesssim 2^{\frac{j(n-1)}{2}}.$
\end{remark}

Now, let us define $\tilde{R}_j^v=\{y:|y-x_0|\leqslant M 2^{-\frac{j}{2}},|\pi_j^v(y-x_0)|\leqslant M 2^{-j}\},$ where $M$ is a large constant to be determined later,
and $\pi_j^v$ represents the projection in the direction of $\xi_j^v$. We also set
$$R_j^v=\{x:|x_0-\Phi_\xi (x,\xi_j^v)|\leqslant M2^{-\frac{j}{2}},|\pi_j^v(x_0-\Phi_\xi (x,\xi_j^v))|\leqslant M2^{-j}\},$$
and $B^*=\bigcup\limits_{2^{-j}\leqslant r}\bigcup\limits_v R_j^v.$ Then
$$
|B^*|\leqslant \sum\limits_{2^{-j}\leqslant r}\sum\limits_{v}|R_j^v|\lesssim \sum\limits_{2^{-j}\leqslant r}\sum\limits_{v}|\tilde{R}_j^v|\lesssim \sum\limits_{2^{-j}\leqslant r} |\Omega_j|\cdot 2^{-j}2^{-\frac{j(n-1)}{2}}\lesssim r.
$$
Consequently,
\begin{eqnarray*}
\int_{B^*}|\F a(x)|dx\leqslant \|\F a\|_{L^2} \cdot |B^*|^\frac12
\lesssim \|a\|_{L^p} \cdot r^\frac12 
\lesssim r^{-n} r^{\frac{n}{p}} r^\frac12=1,
\end{eqnarray*}
where $\frac1{p}=\frac12+\frac{n-1}{2n},$ and for the second inequality, we have utilized Lemma \ref{Lemma One} mentioned above.

Next, we only need to prove that $\int_{^cB^*}|\F a(x)|dx\lesssim 1.$
To achieve this, we will perform a smooth decomposition of cones based on $\Omega_j$ defined earlier, following these steps:
Suppose that $\xi_j^v\in \Omega_j^J,$$I=\{1,2,\cdots,n\}\setminus J,$ then define
$$\Gamma_j^v=\{\xi\ne 0:\sum\limits_{i\in I}\Big[(\frac{\xi}{|\xi|}-\xi_j^v)_i\Big]^2\leqslant  2^{-j-3} \ \text{and}\ \forall i\in J, \frac{|\xi_i|}{|\xi|}\leqslant 2^{1-\frac{j}{2}}\}.$$
We will construct a sequence of functions $\chi_j^v(\xi)$ that are homogeneous of degree zero in $\xi$, have their support within $\Gamma_j^v$, and satisfy
$$
\sum\limits_{v:\xi_j^v\in \Omega_j}\chi_j^{v}(\xi)=1\quad \forall \xi\ne 0 \ \text{and}\ j\geqslant 0,
$$
and
\begin{eqnarray}\label{e1}
|\partial_\xi^\alpha\chi_j^v(\xi)|\lesssim 2^{\frac{|\alpha|j}{2}}|\xi|^{-|\alpha|}.
\end{eqnarray}
In fact, we can construct these as follows: suppose $\varphi\in C^\infty_0(\R)$ satisfies
$\varphi(x)=\begin{cases} 1& |x|\leqslant 1,
\\ 0& |x|>2.\end{cases}$ and $0\leqslant \varphi(x)\leqslant 1.$
 When $\xi_j^v\in \Omega_j^J,I=\{1,2,\cdots,n\}\setminus J,$ let 
 $$\eta^v_j(\xi)=\varphi\big(\sum\limits_{i\in I}[2^{\frac{j}{2}+2}(\frac{\xi}{|\xi|}-\xi_j^v)_i]^2\big)\cdot \prod\limits_{i\in J}\varphi\big(2^{\frac{j}{2}}\frac{|\xi_i|}{|\xi|}\big),$$
and define $\chi_j^v(\xi)=\frac{\eta^v_j(\xi)}{\sum\limits_{\mu:\xi_j^\mu\in \Omega_j} \eta^\mu_j(\xi)}.$ It can be shown that this satisfies the properties mentioned above.

In addition to the cone decomposition, we also need to perform a radial decomposition for the variable $\xi$ as follows. Let
$\Psi(\xi)=\varphi(|\xi|^2)-\varphi(4|\xi|^2),$
and for $j \geqslant 1$, define $\Psi_j(\xi)=\Psi(2^{-j}\xi),$
and let $\Psi_0(\xi)=\varphi(|\xi|^2).$ Then for $\xi \ne 0,$ we have
$$
1=\Psi_0(\xi)+\sum\limits_{j=1}^\infty \sum\limits_{v:\xi_j^v\in \Omega_j} \chi_j^v(\xi)\Psi_j(\xi).
$$
Now for $j\geqslant 0,$ we define
$$
\F_j^vf(x)=\int_{\R^n}e^{2\pi i \Phi(x,\xi)} \sigma(x,\xi)\chi_j^v(\xi)\Psi_j(\xi)\hat{f}(\xi)d\xi.
$$
Its kernel function $K_j^v(x,y)=\int_{\R^n}e^{2\pi i [\Phi(x,\xi)-y\cdot \xi]}a_j^v(x,\xi)d\xi,$ where $a_j^v(x,\xi)=\sigma(x,\xi)\chi_j^v(\xi)\Psi_j(\xi).$
We will prove the following lemma:
\begin{lemma}\label{lem-tan}
Suppose $\xi_j^v\in \Omega_j^J,$ then
\begin{eqnarray}\label{2024.6.4tan1}
\int_{\R^n} |K_j^v(x,y)| dx\lesssim (1+j)^{|J|} 2^{-\frac{j(n-1)}{2}},\ \forall y\in \R^n.
\end{eqnarray}
\begin{eqnarray}\label{2024.6.4tan3}
\int_{\R^n}|K_j^v(x,y)-K_j^v(x,y')|dx\lesssim 2^j|y-y'|\cdot (1+j)^{|J|} 2^{-\frac{j(n-1)}{2}},\ \forall y,y'\in \R^n.
\end{eqnarray}
Furthermore, when $J=\emptyset$, we have
\begin{eqnarray}\label{2024.6.4tan4}
\int_{^c B^*} | K_j^v(x,y)| dx\lesssim \frac{2^{-j}}{r}2^{-\frac{j(n-1)}{2}}, \forall y\in B,2^{-j}\leqslant r.
\end{eqnarray}
\end{lemma}

Assuming the above lemma holds, we now use it to prove that $\int_{^c B^*} |\F a(x)|dx\lesssim 1.$  Indeed, the estimation can be divided into two terms based on whether set $J$ is empty or not:
\begin{eqnarray*}
&&\int_{^c B^*} |\F a(x)|dx=\int_{^c B^*} \Big|\sum\limits_{j=0}^\infty\sum\limits_{J:J\subsetneq \{1,2,\cdots,n\}}\sum\limits_{v:\xi_j^v\in \Omega_j^J} \F_j^v a(x)\Big|dx\\
&\leqslant& \int_{^c B^*}\Big|\sum\limits_{j=0}^\infty\sum\limits_{J\neq\emptyset:J\subsetneq \{1,2,\cdots,n\}} \sum\limits_{v:\xi_j^v\in \Omega_j^J}\F_j^v a(x)\Big|dx+\int_{^c B^*}\Big|\sum\limits_{j=0}^\infty \sum\limits_{v:\xi_j^v\in \Omega_j^\emptyset} \F_j^v a(x)\Big|dx=:I+I\!I.
\end{eqnarray*}
Thus,
\begin{eqnarray*}
I&\leqslant& \sum\limits_{j=0}^\infty\sum\limits_{J\neq\emptyset:J\subsetneq \{1,2,\cdots,n\}}\sum\limits_{v:\xi_j^v\in \Omega_j^J}\int_{\R^n} \Big|\int_{\R^n}K_j^v(x,y)a(y)dy\Big|dx\\
&\overset{\eqref{2024.6.4tan1}}{\lesssim}& \sum\limits_{j=0}^\infty\sum\limits_{J\neq\emptyset:J\subsetneq \{1,2,\cdots,n\}} \sum\limits_{v:\xi_j^v\in \Omega_j^J}(1+j)^{|J|} 2^{-\frac{j(n-1)}{2}}\int_{B} |a(y)|dy\\
&\lesssim& \sum\limits_{j=0}^\infty\sum\limits_{J\neq\emptyset:J\subsetneq \{1,2,\cdots,n\}} (1+j)^{n-1} 2^{-\frac{j(n-1)}{2}}|\Omega_j^J|\\
&\lesssim& \sum\limits_{j=0}^\infty (1+j)^{n-1} 2^{-\frac{j}{2}} \lesssim 1.
\end{eqnarray*}

For the term $I\!I$, we further subdivide it into two parts based on the comparison between $2^{-j}$ and $r$:
\begin{eqnarray*}
I\!I\leqslant \sum\limits_{j:2^{-j}\geqslant r}\int_{^c B^*} \Big| \sum\limits_{v:\xi_j^v\in \Omega_j^\emptyset} \F_j^va(x)\Big|dx+\sum\limits_{j:2^{-j}< r}\int_{^c B^*} \Big| \sum\limits_{v:\xi_j^v\in \Omega_j^\emptyset} \F_j^va(x)\Big|dx=:I\!I_1+I\!I_2.
\end{eqnarray*}
Notice that
\begin{eqnarray*}
I\!I_2&\leqslant& \sum\limits_{j:2^{-j}< r}\sum\limits_{v:\xi_j^v\in \Omega_j^\emptyset}\int_B\int_{^c B^*}  |K_j^v(x,y)|\cdot |a(y)|dxdy\\
&\overset{\eqref{2024.6.4tan4}}{\lesssim}& \sum\limits_{j:2^{-j}< r}\frac{2^{-j}}{r} \int_B |a(y)|dy\cdot2^{-\frac{j(n-1)}{2}}|\Omega_j^\emptyset|\lesssim 1.
\end{eqnarray*}
For the term $I\!I_1$, we utilize the vanishing moment condition of the atom $a$, and obtain
\begin{eqnarray*}
I\!I_1&\leqslant &\sum\limits_{j:2^{-j}\geqslant r}\sum\limits_{v:\xi_j^v\in \Omega_j^\emptyset}\int_{^c B^*} \Big|\int_ B [K_j^v(x,y)-K_j^v(x,x_0)]a(y)dy\Big|dx\\ &\overset{\eqref{2024.6.4tan3}}{\lesssim}&
\sum\limits_{j:2^{-j}\geqslant r}\sum\limits_{v:\xi_j^v\in \Omega_j^\emptyset} 2^j 2^{-\frac{j(n-1)}{2}}\int_ B |y-x_0|\cdot |a(y)|dy\\
&\lesssim& \sum\limits_{j:2^{-j}\geqslant r}2^j r\lesssim 1.
\end{eqnarray*}
Hence, $\int_{^c B^*} |\F a(x)|dx\lesssim 1.$ 
Consequently, when $\sigma\in \S^{-\frac{n-1}{2}}$, we have $\F:H^1(\R^n) \to L^1(\R^n).$ The proof for the adjoint operator $\F^*:H^1(\R^n) \to L^1(\R^n)$ is similar, with the main difference being the replacement of $R_j^v$ with $\{x:|x-\Phi_\xi (x_0,\xi_j^v)|\leqslant M2^{-\frac{j}{2}},|\pi_j^v(x-\Phi_\xi (x_0,\xi_j^v))|\leqslant M2^{-j}\}$, which we omit the detail here.

The proof of part {\rm (ii)} utilizes complex interpolation, as detailed in section {\bf 4.9}, chapter IX of Stein's book \cite{Stein}. We omit this process here for brevity.
What remains is to prove Lemma \ref{lem-tan}.

\section{Proof of Lemma\ref{lem-tan}}
Let $\xi_j^v\in \Omega_j^J,$ $I=\{1,2,\cdots,n\}\setminus J.$ Without loss of generality, we assume $I=\{1,2,\cdots,m\},$\\$J=\{m+1,m+2,\cdots,n\}.$ We can perform a coordinate transformation $\xi=A\eta\Leftrightarrow \eta=A^T\xi,$ where $A=\left(\begin{matrix}A_1&0\\0&A_2\end{matrix}\right)$ is an orthogonal matrix, $A_1$ is an $m$-order orthogonal matrix, $A_2$ is an $(n-m)$-order identity matrix, $\det(A) = 1$, $Ae_1 = \xi_j^v$, and thus $col_1(A) = \xi_j^v$.
According to the coordinate transformation formula, we have $K_j^v(x,y)=\int_{\R^n}e^{2\pi i [\Phi(x,A\eta)-y\cdot A\eta]}a_j^v(x,A\eta)d\eta.$  Below we will estimate $K_j^v(x,y)$, with a declaration: all estimates below assume $j$ is greater than a large constant because if $j$ is less than this constant, the estimates we need are trivial.

Now we write $\Phi(x,A\eta)-y\cdot A\eta$ as $[\Phi_\xi(x,Ae_1)-y]\cdot A\eta+h(\eta),$
where $h(\eta)=[\Phi(x,A\eta)-\Phi_\xi(x,Ae_1)\cdot A\eta].$ For convenience, let $\eta=(\eta_1,\eta')^T,$ where $\eta'=(\eta_2,\cdots,\eta_n)^T.$ 
Now we can express $K_j^v(x,y)$ as
$$
K_j^v(x,y)=\int_{\R^n} e^{2\pi i [\Phi_\xi (x,Ae_1)-y]\cdot A\eta} b_j^v(x,\eta)d\eta,
$$
where $b_j^v(x,\eta)=e^{2\pi i h(\eta)}\chi_j^v(A\eta) \cdot \Psi_j(A\eta)\cdot \sigma(x,A\eta).$
Next, we will prove that for $\alpha=(\alpha_1,\alpha_2,\cdots,\alpha_n),$ it holds
\begin{eqnarray}\label{2024.7.1e3}
\Big|\partial_\eta^\alpha e^{2\pi i h(\eta)}\Big|+\Big|\partial_\eta^\alpha[\chi_j^v(A\eta)]\Big|+\Big|\partial_\eta^\alpha[\Psi_j(A\eta)]\Big|\lesssim_\alpha  2^{-\frac{j}{2}\alpha_1-\frac{j}{2}|\alpha|}.
\end{eqnarray}

First, we show that 
\begin{eqnarray}\label{eq2024.7.1}
\Big|\partial_\eta^\alpha e^{2\pi i h(\eta)}\Big|\lesssim_\alpha  2^{-\frac{j}{2}\alpha_1-\frac{j}{2}|\alpha|}.\end{eqnarray}

In fact, when $\xi=A\eta\in \supp(b_j^v(x,\cdot)),$ since $\xi\in \supp (\Psi_j(\cdot)),$ it follows that $2^{j-1}\leqslant |\xi|=|\eta|\leqslant 2^{j+1}.$ 
Simultaneously, we have $|\eta'|\leqslant |\eta-|\eta|e_1|= |\frac{\eta}{|\eta|}-e_1||\eta|=|\frac{\xi}{|\xi|}-\xi_j^v||\xi|\lesssim 2^{\frac{j}{2}}.$

Notice also that $\Phi(x,\xi)$ is homogeneous of degree one in $\xi$, i.e., $\Phi(x,\delta\xi)=\delta\Phi(x,\xi),$ and by differentiation with respect to $\delta$ and setting $\delta = 1$, we obtain $\Phi_\xi(x,\xi)\cdot \xi=\Phi(x,\xi).$
Therefore, $h(\eta)=[\Phi_\xi(x,A\eta)-\Phi_\xi(x,Ae_1)]\cdot A\eta,$ which implies $h(\eta_1,0,\cdots,0)=0,$ and noticing
$\triangledown_\eta h=A^T\big(\Phi_\xi(x,A\eta)-\Phi_\xi(x,Ae_1)\big)$, it follows that $\triangledown_\eta h(\eta_1,0,\cdots,0)=0,$ Further, for $\alpha=(\alpha_1,\alpha_2,\cdots,\alpha_n),$ it holds
\begin{eqnarray}\label{2024.7.1e2}
\partial_\eta^\alpha h(\eta_1,0,\cdots,0)=0,\quad \text{where} \ 0\leqslant \alpha_2+\alpha_3\cdots+\alpha_n\leqslant 1.
\end{eqnarray}

Since $h(\eta)$ is homogeneous of degree one, i.e., $h(\delta\eta)=\delta h(\eta),$ it follows that $\delta^{\alpha}(\partial_\eta^\alpha h)(\delta\eta)=\delta(\partial_\eta^\alpha h)(\eta),$
 and taking $\delta=\frac1{|\eta|},$ we obtain
$|\partial_\eta^\alpha h(\eta)|\lesssim_\alpha \frac1{|\eta|^{|\alpha|-1}}.$ To prove inequality \eqref{eq2024.7.1}, we proceed by examining three distinct cases.

Case (1): $\alpha_2+\alpha_3\cdots+\alpha_n\geqslant 2$, we have $|\partial_\eta^\alpha h(\eta)|\lesssim 2^{-j(|\alpha|-1)} \lesssim  2^{-\frac{j}{2}\alpha_1-\frac{j}{2}|\alpha|}.$

For the other two cases, we employ Taylor's expansion, asserting the existence of $0\leqslant \theta\leqslant 1,$ such that 
\begin{eqnarray*}
\partial_\eta^\alpha h (\eta)&=&\partial_\eta^\alpha h(\eta_1,0,\cdots,0)+\eta'\cdot \triangledown_{\eta'}\partial_\eta^\alpha h(\eta_1,0,\cdots,0)
+\sum\limits_{\beta=(0,\beta_2,\cdots,\beta_n):|\beta|=2}\frac1{\beta!}\eta^\beta \cdot \partial_\eta^{\alpha+\beta} h (\eta_1,\theta\eta)\\
&\overset{\eqref{2024.7.1e2}}{=}&\eta'\cdot \triangledown_{\eta'}\partial_\eta^\alpha h(\eta_1,0,\cdots,0)+\sum\limits_{\beta=(0,\beta_2,\cdots,\beta_n):|\beta|=2}\frac1{\beta!}\eta^\beta \cdot \partial_\eta^{\alpha+\beta} h (\eta_1,\theta\eta')
\end{eqnarray*}

Case (2): $\alpha_2+\alpha_3\cdots+\alpha_n= 1$, then  we have  $|\partial_\eta^\alpha h(\eta)|\lesssim 2^{\frac{j}{2}}2^{-j|\alpha|}+2^j 2^{-j(|\alpha|+1)} \lesssim  2^{-\frac{j}{2}\alpha_1-\frac{j}{2}|\alpha|}.$

Case (3): $\alpha_2+\alpha_3\cdots+\alpha_n= 0$, then by \eqref{2024.7.1e2}, we have $\partial_\eta^\alpha h (\eta)=\sum\limits_{\beta=(0,\beta_2,\cdots,\beta_n):|\beta|=2}\frac1{\beta!}\eta^\beta \cdot \partial_\eta^{\alpha+\beta} h (\eta_1,\theta\eta').$
Hence, $|\partial_\eta^\alpha h(\eta)|\lesssim 2^j 2^{-j(|\alpha|+1)} \lesssim  2^{-\frac{j}{2}\alpha_1-\frac{j}{2}|\alpha|}.$ Combining all three scenarios, we thereby conclude that inequality \eqref{eq2024.7.1} is indeed valid.

We also note that when $\xi=A\eta\in\supp(a_j^v(x,\cdot))$, inequality \eqref{e1} yields:
$
|\partial_\eta^\alpha\widetilde{\chi_j^v}(\eta)|\lesssim 2^{\frac{|\alpha|j}{2}}|\eta|^{-|\alpha|}\sim 2^{-\frac{|\alpha|j}{2}},
$
where $\widetilde{\chi_j^v}(\eta)=\chi_j^v(A\eta)=\frac{\eta^v_j(A\eta)}{\sum\limits_{\mu:\xi_j^\mu\in \Omega_j} \eta^\mu_j(A\eta)}.$
Furthermore, this can be refined to $|\partial_\eta^\alpha\widetilde{\chi_j^v}(\eta)| \lesssim   2^{-\frac{j}{2}\alpha_1-\frac{j}{2}|\alpha|}.$
In fact, since $A=\left(\begin{matrix}A_1&0\\0&A_2\end{matrix}\right)$, $A_1$ is an $m$-order orthogonal matrix, $A_2$ is an $(n-m)$-order identity matrix, we have
\begin{eqnarray*}
\eta^v_j(A\eta)&=&\varphi\big(\sum\limits_{i\in I}[2^{\frac{j}{2}+2}(\frac{A\eta}{|\eta|}-Ae_1)_i]^2\big)\cdot \prod\limits_{i\in J}\varphi\big(2^{\frac{j}{2}}\frac{(A\eta)_i}{|\eta|}\big)\\
&=&\varphi\big(\sum\limits_{i\in I}[2^{\frac{j}{2}+2}(\frac{\eta}{|\eta|}-e_1)_i]^2\big)\cdot \prod\limits_{i\in J}\varphi\big(2^{\frac{j}{2}}\frac{\eta_i}{|\eta|}\big).\end{eqnarray*}

Additionally, we observe that $\frac{\partial}{\partial{\eta_1}}\frac{\eta_1}{|\eta|}=\frac{1}{|\eta|}-\frac{\eta_1^2}{|\eta|^3}=\frac{|\eta'|^2}{|\eta|^3},$  and $\frac{\partial}{\partial{\eta_1}}\frac{\eta'}{|\eta|}=-\frac{\eta'\eta_1}{|\eta|^3},$
thus $|\frac{\partial}{\partial{\eta_1}}\frac{\eta}{|\eta|}|\lesssim \frac{2^{\frac{j}{2}}2^j}{2^{3j}}\sim \frac{2^{-\frac{j}{2}}}{|\eta|}.$ 
This implies that each differentiation of $\eta^v_j(A\eta)$ with respect to $\eta_1$ results in a decay rate of $\frac{1}{|\eta|}$, without producing a coefficient of $2^{\frac{j}{2}}$.
Consequently,
\begin{eqnarray*}
|\partial_\eta^\alpha[\eta_j^v(A\eta)]| &\lesssim& 2^{\frac{(|\alpha|-\alpha_1)j}{2}}|\eta|^{-|\alpha|} \lesssim 2^{-\frac{j}{2}\alpha_1-\frac{j}{2}|\alpha|}.\end{eqnarray*}

Hence,$\Big|\partial_\eta^\alpha[\chi_j^v(A\eta)]\Big|\lesssim_\alpha 2^{-\frac{j}{2}\alpha_1-\frac{j}{2}|\alpha|}.$
It is evident that $\Big|\partial_\eta^\alpha[\Psi_j(A\eta)]\Big|\lesssim_\alpha 2^{-j|\alpha|}\lesssim 2^{-\frac{j}{2}\alpha_1-\frac{j}{2}|\alpha|}.$
Therefore, inequality \eqref{2024.7.1e3} is proven.

The next critical step is to estimate $\partial_\eta^\alpha [\sigma(x,A\eta)]$  and to provide an estimate for the kernel $K_j^v(x,y)$. For this purpose, we require a further decomposition, as detailed below:

Let $\varphi\in C^\infty_0(\R)$ satisfy
$
\varphi(t)=\begin{cases} 1& |t|\leqslant 1,
\\ 0& |t|>2.\end{cases}$
For $0\leqslant l\leqslant j,$ define
$$\varphi_{l,j}(t)=\begin{cases}\varphi(2^{l-j-1}t)-\varphi(2^{l-j}t), &l<j \\ \varphi(2^{-1}t),&l=j.\end{cases}$$

For $0\leqslant \ell=(l_1,l_2,l_3,\cdots,l_{n})\leqslant j,$ that is, $0\leqslant l_i\leqslant j,\forall 1\leqslant i\leqslant n,$ we define
$\delta_{\ell}(\xi) = \prod\limits_{i=1}^{n}\varphi_{l_i,j}(\xi_i).$
It can be proven that $|\partial_{\xi}^\alpha \delta_\ell(\xi)|\lesssim 2^{\ell\cdot \alpha-j|\alpha|},$ where $\alpha=(\alpha_1,\alpha_2,\cdots,\alpha_n)$, $\ell\cdot \alpha=\alpha_1 l_1+\alpha_2 l_2+\cdots+\alpha_n l_n$ and $|\alpha|=\alpha_1+\alpha_2+\cdots+\alpha_n.$

It is evident that when $|\xi|\leqslant 2^{j+1},$ $\sum\limits_{\ell:0\leqslant \ell\leqslant j}\delta_{\ell}(\xi)=1,$ Therefore, we further decompose $K_j^v(x,y)$ as follows:
\begin{eqnarray*}
&&K_j^v(x,y)
=\sum\limits_{\ell:0\leqslant \ell\leqslant j}\int_{\R^n} e^{2\pi i [\Phi_\xi (x,Ae_1)-y]\cdot A\eta} \delta_\ell(A\eta) b_j^v(x,\eta)d\eta,
\end{eqnarray*}
where $b_j^v(x,\eta)=e^{2\pi i h(\eta)}\chi_j^v(A\eta) \cdot \Psi_j(A\eta)\cdot \sigma(x,A\eta).$

When $\eta\in \supp(\delta_\ell(A\cdot) b_j^v(x,\cdot))$, we analyze the range of each component of $\xi = A\eta$ below.\\
{\rm (i)} $2^{j-1}\leqslant |\xi|=|\eta|\leqslant 2^{j+1},$\quad $|\eta'|\lesssim 2^{\frac{j}{2}}.$\\
{\rm (ii)} $\xi\in \supp(\chi_j^v(\cdot))\Rightarrow$$\begin{cases}[2^{\frac{j}{2}+2}(\frac{\xi}{|\xi|}-\xi_j^v)_i]^2\leqslant 2\Rightarrow |\frac{\xi_i}{|\xi|}-(\xi_j^v)_i|\leqslant 2^{-\frac{j}{2}-1},&\forall\ i\in I \\ 2^\frac{j}{2}\frac{|\xi_i|}{|\xi|}\leqslant 2\Rightarrow \frac{|\xi_i|}{|\xi|}\leqslant 2^{-\frac{j}{2}+1},&\forall\ i\in J.\end{cases}$\\
{\rm (iii)} $\xi_j^v\in \Omega_j^J\Rightarrow \begin{cases}|(\xi_j^v)_i|\geqslant 2^{-\frac{j}{2}},&\forall\ i\in I \\ (\xi_j^v)_i=0,&\forall\ i\in J.\end{cases}$\\
{\rm (iv)} $\xi\in \supp(\delta_\ell(\cdot)) \Rightarrow  \forall 1\leqslant i\leqslant n, \begin{cases}2^{j-l_i}\leqslant |\xi_i|\leqslant 2^{j-l_i+2} \Rightarrow 2^{-l_i-1}\leqslant \frac{|\xi_i|}{|\xi|}\leqslant 2^{-l_i+3},&\forall\ 
l_i<j\\ |\xi_i|\leqslant 4\Rightarrow \frac{|\xi_i|}{|\xi|} \leqslant 2^{-j+3},&\forall\ l_i=j.\end{cases}$

From {\rm (ii)} and {\rm (iii)},  when $i\in I$, $\frac{|\xi_i|}{|\xi|}\sim |(\xi_j^v)_i|\geqslant 2^{-\frac{j}{2}}$. When $j$ is sufficiently large, according to {\rm (iv)}, it must be that \begin{eqnarray}\label{2024.7.5ie0}
l_i<j \ \text{ and } \ |(\xi_j^v)_i|\sim\frac{|\xi_i|}{|\xi|}\sim 2^{-l_i}\Rightarrow 2^{-l_i}\gtrsim 2^{-\frac{j}{2}}.
\end{eqnarray}

From {\rm (ii)} and {\rm (iv)}, when $i\in J$, we have $2^{-l_i}\lesssim 2^{-\frac{j}{2}}.$

We can now estimate the partial derivatives of $\sigma(x,A\eta)$ and $\delta_\ell(A\eta)$ with respect to $\eta$. Since $\sigma(x,\eta)\in \S^{-\frac{n-1}{2}}$, it follows that
\begin{eqnarray*}
\Big|\frac{\partial}{\partial{\eta_1}} [\sigma(x,A\eta)]\Big|&=&|(\triangledown_{\xi}\sigma)(x,\xi)\cdot \xi_j^v|
\overset{\eqref{2024.7.5ie0}}{\lesssim} 2^{-j\frac{n-1}{2}} \sum\limits_{i\in I}\frac{2^{-l_i}}{1+2^{j-l_i}}\sim  2^{-j\frac{n-1}{2}}2^{-j},
\end{eqnarray*}
\begin{eqnarray*}
\Big|\frac{\partial}{\partial{\eta_1}} [\delta_\ell(A\eta)]\Big|&=&|(\triangledown_{\xi}\delta_\ell)(\xi)\cdot \xi_j^v|
\overset{\eqref{2024.7.5ie0}}{\lesssim} \sum\limits_{i\in I}2^{l_i-j}\cdot 2^{-l_i}\sim  2^{-j},
\end{eqnarray*}

Notice that $\triangledown_{\eta} [\sigma(x,A\eta)]=A^T[(\triangledown_\xi\sigma)(x,A\eta)],$ thus for $i\in I\setminus\{1\}$, we have
\begin{eqnarray*}|\frac{\partial}{\partial \eta_i} [\sigma(x,A\eta)]|\leqslant \Big|\sum\limits_{k=1}^m (\frac{\partial}{\partial \xi_k}\sigma)(x,\xi)\Big|\lesssim 2^{-j\frac{n-1}{2}}\sum\limits_{k\in I}\frac1{1+2^{j-l_k}}\overset{\eqref{2024.7.5ie0}}{\lesssim} 2^{-j\frac{n-1}{2}} 2^{-\frac{j}{2}}.
\end{eqnarray*}
\begin{eqnarray*}|\frac{\partial}{\partial \eta_i} [\delta_\ell(A\eta)]|\leqslant \Big|\sum\limits_{k=1}^m (\frac{\partial}{\partial \xi_k}\delta_\ell)(\xi)\Big|\lesssim \sum\limits_{k\in I} 2^{l_k-j}\overset{\eqref{2024.7.5ie0}}{\lesssim} 2^{-\frac{j}{2}}.
\end{eqnarray*}

When $i \in J$ and if $l_i < j$, then
\begin{eqnarray*}|\frac{\partial}{\partial \eta_i} [\sigma(x,A\eta)]|=\Big|\big(\frac{\partial}{\partial \xi_i}\sigma\big)(x,\xi)\Big|\lesssim 2^{-j\frac{n-1}{2}} 2^{l_i-j},
\end{eqnarray*}
and
\begin{eqnarray*}|\frac{\partial}{\partial \eta_i} [\delta_\ell(A\eta)]|=\Big|\big(\frac{\partial}{\partial \xi_i}\delta_\ell\big)(\xi)\Big|\lesssim 2^{l_i-j}.
\end{eqnarray*}
If $l_i = j$, the above inequalities hold trivially.

Similarly, for $\alpha=(\alpha_1,\alpha_2,\cdots,\alpha_n),$ we can further obtain
\begin{eqnarray*}
\Big|\partial_\eta^\alpha [\sigma(x,A\eta)]\Big|\lesssim  2^{-j\frac{n-1}{2}}\cdot 2^{-\frac{j}{2}\alpha_1-\frac{j}{2}\sum\limits_{i\in I}\alpha_i-\sum\limits_{i\in J}\alpha_i(j-l_i)},
\end{eqnarray*}
and
\begin{eqnarray*}
\Big|\partial_\eta^\alpha [\delta_\ell(A\eta)]\Big|\lesssim   2^{-\frac{j}{2}\alpha_1-\frac{j}{2}\sum\limits_{i\in I}\alpha_i-\sum\limits_{i\in J}\alpha_i(j-l_i)}.
\end{eqnarray*}
Let $L=I-2^{2j}\frac{\partial^2}{\partial \eta_1^2}-2^j\sum\limits_{i\in I} \frac{\partial^2}{\partial \eta_i^2}-\sum\limits_{i\in J} 2^{2(j-l_i)}\frac{\partial^2}{\partial \eta_i^2}.$
Based on the estimates above, we have
$$|L^N[\delta_\ell(A\eta)b_j^v(x,\eta)]|\lesssim_N 2^{-\frac{j(n-1)}{2}}.$$
Since $L^N(e^{2\pi i [\Phi_\xi (x,Ae_1)-y]\cdot A\eta})=C_j(x,y)^N e^{2\pi i [\Phi_\xi (x,Ae_1)-y]\cdot A\eta},$ where $C_j(x,y)=1+4\pi^2 \big[2^{2j}|\tilde{x}_1|^2+ 2^j\sum\limits_{i\in I} |\tilde{x}_i|^2+ \sum\limits_{i\in J} 2^{2(j-l_i)}|\tilde{x}_i|^2\big]$ and $\tilde{x}=A^T[\Phi_\xi (x,Ae_1)-y].$
Therefore,
\begin{eqnarray*}
K_{j,v}(x,y)&=&\sum\limits_{\ell:0\leqslant \ell \leqslant j \atop \forall\ i\in I, 2^{-l_i}\sim |(\xi_j^v)_i| }\frac1{C_j(x,y)^N}\int_{\R^n} L^N\big(e^{2\pi i [\Phi_\xi (x,Ae_1)-y]\cdot A\eta}\big) \delta_\ell(A\eta) b_j^v(x,\eta)d\eta\\
&=&\sum\limits_{\ell:0\leqslant \ell \leqslant j \atop \forall\ i\in I, 2^{-l_i}\sim |(\xi_j^v)_i| }\frac1{C_j(x,y)^N}\int_{\R^n} e^{2\pi i [\Phi_\xi (x,Ae_1)-y]\cdot A\eta} L^N\big[\delta_\ell(A\eta)b_j^v(x,\eta)\big]d\eta.
\end{eqnarray*}
Choosing $N>\frac{n}{2}$, and by $\supp[\delta_\ell(A\cdot) b_j^v(x,\cdot)]\subseteq \{\eta:|\eta|\lesssim 2^{j},|\eta'|\lesssim 2^{\frac{j}{2}},|\eta_i|\lesssim 2^{j-l_i}, \forall i\in J\}$, it follows that
\begin{eqnarray}\label{e2}
|K_{j,v}(x,y)|\lesssim \sum\limits_{\ell:0\leqslant \ell \leqslant j \atop \forall\ i\in I, 2^{-l_i}\sim |(\xi_j^v)_i| }\frac1{|C_{j}(x,y)|^N} 2^{-\frac{j(n-1)}{2}} 2^{j}\cdot 2^{\frac{j(m-1)}{2}}\cdot 2^{\sum\limits_{i\in J} (j-l_i)},
\end{eqnarray}
therefore
\begin{eqnarray*}
&&\int_{\R^n} |K_{j,v}(x,y)| dx\lesssim \sum\limits_{\ell:0\leqslant \ell \leqslant j \atop \forall\ i\in I, 2^{-l_i}\sim |(\xi_j^v)_i| }2^{\frac{j(n-m)}{2}-\sum\limits_{i\in J} l_i}\cdot\int_{\R^n}\frac{2^j}{|C_{j}(x,y)|^N}dx\\
&\lesssim &\sum\limits_{\ell:0\leqslant \ell \leqslant j \atop \forall\ i\in I, 2^{-l_i}\sim |(\xi_j^v)_i| }2^{\frac{j(n-m)}{2}-\sum\limits_{i\in J} l_i}\cdot\\&&\int_{\R^n}\frac{2^j}{\big(1+ 2^{2j}|[A^T(u-y)]_1|^{2}+ 2^{j}\sum\limits_{i\in I}|[A^T(u-y)]_i|^{2}+ \sum\limits_{i\in J}2^{2j-2l_i} |[A^T(u-y)]_i|^{2}\big)^N} \frac1{\left|\det\left(\frac{\partial^2\Phi}{\partial x_i\partial \xi_j}\right)\right|}du\\
&\lesssim&\sum\limits_{\ell:0\leqslant \ell \leqslant j \atop \forall\ i\in I, 2^{-l_i}\sim |(\xi_j^v)_i| }2^{\frac{j(n-m)}{2}-\sum\limits_{i\in J} l_i}\int_{\R^n}\frac{2^j}{(1+2^{2j}|v_1|^{2}+2^{j}\sum\limits_{i\in I\setminus\{1\}} |v_i|^{2}+\sum\limits_{i\in J}2^{2j-2l_i} |v_i|^{2})^{N}} dv\\
&\lesssim& \sum\limits_{\ell:0\leqslant \ell \leqslant j \atop \forall\ i\in I, 2^{-l_i}\sim |(\xi_j^v)_i| }2^{\frac{j(n-m)}{2}-\sum\limits_{i\in J} l_i} \cdot 2^{-\frac{j(m-1)}{2}}\cdot 2^{-(n-m)j+\sum\limits_{i\in J} l_i}\\
&\lesssim& (1+j)^{n-m} 2^{-\frac{j(n-1)}{2}}=(1+j)^{|J|} 2^{-\frac{j(n-1)}{2}}.
\end{eqnarray*}

Hence, the inequality \eqref{2024.6.4tan1} in Lemma \ref{lem-tan} holds.

Furthermore, we have
\begin{eqnarray*}
|\triangledown_y K_j^v(x,y)|&\lesssim&\sum\limits_{\ell:0\leqslant \ell \leqslant j \atop \forall\ i\in I, 2^{-l_i}\sim |(\xi_j^v)_i| }\Big|\int_{\R^n} e^{2\pi i [\Phi_\xi (x,Ae_1)-y]\cdot A\eta} \cdot 2\pi i \cdot A\eta b_j^v(x,\eta)\delta_\ell(A\eta)d\eta\Big|\\
&=&\sum\limits_{\ell:0\leqslant \ell \leqslant j \atop \forall\ i\in I, 2^{-l_i}\sim |(\xi_j^v)_i| }\Big|\frac1{C_j(x,y)^N}\int_{\R^n} L^N\big(e^{2\pi i [\Phi_\xi (x,Ae_1)-y]\cdot A\eta}\big)\cdot 2\pi i\cdot A\eta b_j^v(x,\eta)\delta_\ell(A\eta)d\eta\Big|\\
&=&\sum\limits_{\ell:0\leqslant \ell \leqslant j \atop \forall\ i\in I, 2^{-l_i}\sim |(\xi_j^v)_i| }\Big|\frac1{C_j(x,y)^N}\int_{\R^n} e^{2\pi i \cdot[\Phi_\xi (x,Ae_1)-y]\cdot A\eta} L^N\big(2\pi i A\eta b_j^v(x,\eta)\delta_\ell(A\eta)d\eta\big)\Big|\\
&\lesssim&\sum\limits_{\ell:0\leqslant \ell \leqslant j \atop \forall\ i\in I, 2^{-l_i}\sim |(\xi_j^v)_i| }\frac1{|C_{j}(x,y)|^N} 2^{-\frac{j(n-1)}{2}} 2^{2j}\cdot 2^{\frac{j(m-1)}{2}}\cdot 2^{\sum\limits_{i\in J} (j-l_i)}.
\end{eqnarray*}
Consequently, 
\begin{eqnarray*}
\int_{\R^n} |\triangledown_y K_j^v(x,y)| dx\lesssim  2^j\cdot (1+j)^{|J|}\cdot 2^{-\frac{j(n-1)}{2}}.
\end{eqnarray*}
Hence, we obtain
$$
\int_{\R^n} |K_j^v(x,y)-K_j^v(x,y')| dx\lesssim  2^j|y-y'|\cdot  (1+j)^{|J|} 2^{-\frac{j(n-1)}{2}}\quad\forall y,y'\in \R^n.
$$
This confirms the inequality \eqref{2024.6.4tan3} of Lemma \ref{lem-tan} holds.

Next, we will prove the inequality \eqref{2024.6.4tan4}, which is to show that under the condition $J = \emptyset$,
$\int_{^c B^*} | K_j^v(x,y)| dx\lesssim \frac{2^{-j}}{r}2^{-\frac{j(n-1)}{2}},$ where $y\in B,2^{-j}\leqslant r.$ 
The crucial step is to demonstrate the following inequality:
\begin{eqnarray}\label{2024.7.5ie5}
2^{\frac{j}{2}}|y-\Phi_\xi (x,\xi_j^v)|+2^{j}|\pi_j^v(y-\Phi_\xi (x,\xi_j^v))|\gtrsim \big(2^j r\big)^\frac12.
\end{eqnarray}

We first assume this inequality holds. Then, according to \eqref{e2}, we have 
$| K_j^v(x,y)|\lesssim  \frac{2^j }{C_j(x,y)^{N-1}}\frac{2^{-j}}{r}.$
This implies that $\int_{^c B^*} | K_j^v(x,y)| dx\lesssim \frac{2^{-j}}{r}2^{-\frac{j(n-1)}{2}}.$
This confirms the inequality \eqref{2024.6.4tan4} of Lemma \ref{lem-tan} holds.

We now proceed to prove the inequality \eqref{2024.7.5ie5}. Note that $B^*=\bigcup\limits_{k:2^{-k}\leqslant r}\bigcup\limits_v R_j^v,$ thus $^c B^*=\bigcap\limits_{k:2^{-k}\leqslant r}\bigcap\limits_\mu  {^c R_k^\mu}.$
Choose $k$ such that $2^{-k}\leqslant r<2^{-k+1}$. By the construction of $\Omega_k$ and that $J = \emptyset$, there exists a $\xi_k^\mu$ for which $|\xi_j^v-\xi_k^\mu|\leqslant 2^{-\frac{k}{2}-2}.$ When $x\in \lc B^*$, it implies $x\in {^cR_k^\mu},$
thereby giving $$2^{\frac{k}{2}}|x_0-\Phi_\xi (x,\xi_k^\mu)|+2^{k}|\pi_k^\mu(x_0-\Phi_\xi (x,\xi_k^\mu))|\geqslant M,$$ where $M$ is a large constant. Since $|y-x_0|\leqslant r<2^{-k+1},$ it follows that
$$2^{\frac{k}{2}}|y-\Phi_\xi (x,\xi_k^\mu)|+2^{k}|\pi_k^\mu(y-\Phi_\xi (x,\xi_k^\mu))|\geqslant \frac{M}{2},$$
thus either $|y-\Phi_\xi (x,\xi_k^\mu)|\geqslant \frac{M}{4} 2^{-\frac{k}{2}}$ or $|\pi_k^\mu(y-\Phi_\xi (x,\xi_k^\mu))|\geqslant \frac{M}{4} 2^{-k}.$

When $|y-\Phi_\xi (x,\xi_k^\mu)|\geqslant \frac{M}{4} 2^{-\frac{k}{2}}$, note that $|\Phi_\xi (x,\xi_k^\mu)-\Phi_\xi (x,\xi_j^v)|\lesssim |\xi_k^\mu-\xi_j^v|\lesssim 2^{-\frac{k}{2}},$
thus $|y-\Phi_\xi (x,\xi_j^v)|\gtrsim 2^{-\frac{k}{2}}.$ Therefore, the inequality \eqref{2024.7.5ie5} holds.

When $|\pi_k^\mu(y-\Phi_\xi (x,\xi_k^\mu))|\geqslant \frac{M}{4}  2^{-k}$, note that
\begin{eqnarray*}
&&|\pi_k^\mu\big[\Phi_\xi (x,\xi_j^v)-\Phi_\xi (x,\xi_k^\mu)\big]|=|\langle\Phi_\xi (x,\xi_j^v)-\Phi_\xi (x,\xi_k^\mu),\xi_k^\mu \rangle|\\
&\leqslant&\Big|\langle \Phi_{\xi\xi} (x,\xi_k^\mu),(\xi_j-\xi_k^\mu)\otimes\xi_k^\mu\rangle\Big|\\
&+&\Big|\big\langle \sum\limits_{\alpha:|\alpha|=2}\partial_{\xi}^\alpha \Phi_\xi \big(x,\xi_k^\mu+\theta (\xi_j^v-\xi_k^\mu)\big)\cdot(\xi_j-\xi_k^\mu)^\alpha, \xi_k^\mu \big\rangle\Big|,\quad \text{where}\ 0<\theta<1.\\
&\lesssim&\Big|\langle \sum\limits_{i=1}^n\partial_{\xi_i}\Phi_\xi (x,\xi_k^\mu)(\xi_k^\mu)_i, \xi_j-\xi_k^\mu \rangle\Big|+2^{-k}\\
&=&2^{-k},
\end{eqnarray*}
here the zero-homogeneity of $\Phi_{\xi}$ with respect to $\xi$ is used, thus the directional derivative in the direction of $\xi_k^\mu$  is 0.
Consequently, $|\pi_k^\mu(y-\Phi_\xi (x,\xi_j^v))|\geqslant \frac{M}{8}  2^{-k}.$ Now let $\theta=\frac{|\pi_k^\mu(y-\Phi_\xi (x,\xi_j^v))|}{|y-\Phi_\xi (x,\xi_j^v)|}.$\\
Case (1): $\theta>  2^{1-\frac{k}{2}}$, then
\begin{eqnarray*}
|\{A^T[y-\Phi_\xi (x,\xi_j^v)]\}_1|&=&|\langle y-\Phi_\xi (x,\xi_j^v),\xi_j^v\rangle|\\
&\geqslant&-|\langle y-\Phi_\xi (x,\xi_j^v),\xi_j^v-\xi_k^\mu\rangle|+|\langle y-\Phi_\xi (x,\xi_j^v),\xi_k^\mu\rangle|\\
&\geqslant&-2^{-\frac{k}{2}}|y-\Phi_\xi (x,\xi_j^v)|+|\pi_k^\mu(y-\Phi_\xi (x,\xi_j^v))|\\
&=&(-\frac{2^{-\frac{k}{2}}}{\theta}+1)|\pi_k^\mu(y-\Phi_\xi (x,\xi_j^v))|\gtrsim 2^{-k}.
\end{eqnarray*}
Case (2): $\theta\leqslant 2^{1-\frac{k}{2}}$, $|y-\Phi_\xi (x,\xi_j^v)|=\frac{|\pi_k^\mu(y-\Phi_\xi (x,\xi_j^v))|}{\theta}\gtrsim \frac{2^{-k}}{2^{-\frac{k}{2}}}=2^{-\frac{k}{2}}.$

In both cases, we can deduce that $2^{\frac{j}{2}}|A^T[y-\Phi_\xi (x,\xi_j^v)]|+2^{j}|\{A^T[y-\Phi_\xi (x,\xi_j^v)]\}_1|\gtrsim 2^{\frac{j-k}{2}}$, thereby verifying the inequality \eqref{2024.7.5ie5}. This concludes the proof of Lemma \ref{lem-tan}.

 \bigskip
 
\noindent  Chaoqiang Tan, Department of Mathematics, Shantou
University, Shantou, 515063, P. R. China.

\noindent {\it E-mail address}: \texttt{cqtan@stu.edu.cn }

\smallskip
\noindent  Zipeng Wang, Department of Mathematics, School of Science, Westlake University, Hangzhou, Zhejiang 310030, P. R. China.

\noindent {\it E-mail address}: \texttt{wangzipeng@westlake.edu.cn }

\end{document}